\documentclass[fleqn]{amsart}

\usepackage[section]{placeins} 
\usepackage{graphicx}

\usepackage{ctable} 
\newcommand{\otoprule}{\midrule[\heavyrulewidth]}%
\usepackage{multicol}

\usepackage{cite}

\usepackage[ocgcolorlinks]{hyperref}
\hypersetup{plainpages=false, linktocpage=true, colorlinks=false,
            breaklinks=true, linkcolor=black, menucolor=black,
			urlcolor=black, citecolor=black}

\hypersetup{pdfauthor={Andreas Wachtel and Gabriel R. Barrenechea}}
\hypersetup{pdfsubject={Numerical Analysis}}
\hypersetup{pdftitle={The inf-sup stability of the lowest order Taylor-Hood pair on anisotropic meshes}}
\hypersetup{pdfkeywords={anisotropic mesh, Taylor, Hood, LBB condition}}

\usepackage{mathtools}
\usepackage{amsfonts}
\usepackage{amsthm}

\newtheorem{lemma}{Lemma}
\newtheorem{theorem}[lemma]{Theorem}

\newtheorem{remark}[lemma]{Remark}

\newtheorem{hypo}{Hypothesis}

\let\originalleft\left%
\let\originalright\right%
\renewcommand{\left}{\mathopen{}\mathclose\bgroup\originalleft}%
\renewcommand{\right}{\aftergroup\egroup\originalright}%

\DeclareMathOperator*{\Div}{div}
\renewcommand{\b}[1]{{\boldsymbol{#1}}}
\providecommand{\RR}{\mathbb{R}}

\providecommand{\spacedtext}[2][\quad]{#1\text{#2}#1}
\newcommand{\deep}[1]{\raisebox{-0.625ex}{#1}}
\newcommand{\deeper}[1]{\raisebox{-0.75ex}{#1}}

\newcommand{\set}[1]{\left\{#1\right\}}
\newcommand{\abs}[1]{\left\lvert{#1}\right\rvert}
\newcommand{\norm}[2]{\left\lVert{#1}\right\rVert_{#2}}
\newcommand{\seminorm}[2]{\left\lvert{#1}\right\rvert_{#2}}
\newcommand{\restrict}[2]{\left.{#1}\right|_{#2}}
\newcommand{\meas}[1]{\abs{#1}}%

\providecommand{\Hseminorm}[2]{\seminorm{#1}{1,#2}}
\providecommand{\Hnorm}[2]{\norm{#1}{1,#2}}
\providecommand{\Lnorm}[2]{\norm{#1}{0,#2}}
\providecommand{\Lprod}[2]{\ensuremath{\left({#1},{#2}\right)}}

\providecommand{\cl}[1]{\overline{#1}}%

\usepackage[mathscr]{euscript}
\providecommand{\PO}{{\mathscr{P}}}
\providecommand{\MP}{M_{\PO}}%
\providecommand{\VP}{\b{V}_{\PO}}%
\providecommand{\bv}{\b{v}}%
\providecommand{\bc}{\b{c}}%
\providecommand{\VX}{\mathcal{X}}

\usepackage{stmaryrd}   
\SetSymbolFont{stmry}{bold}{U}{stmry}{m}{n}
\newcommand{\jump}[1]{\ensuremath{\left\llbracket{#1}\right\rrbracket}}
\newcommand{\avge}[2]{\ensuremath{\left<{#1}\right>_{#2}}}

\begin{document}
%
\providecommand{\coloneqq}{:=}
\newcommand{\bydef}{\coloneqq}%
\title[Taylor--Hood pairs in anisotropic meshes]{The inf-sup stability of the lowest order Taylor-Hood pair on anisotropic meshes}
\author{Gabriel R.\ Barrenechea}
\address{Department of Mathematics and Statistics, University of Strathclyde, 26 Richmond Street, Glasgow G1 1XH, Scotland}
\email{gabriel.barrenechea@strath.ac.uk}
\author{Andreas Wachtel}
\address[corresponding author]{Department of Mathematics, ITAM, R\'io Hondo 1, Ciudad de M\'exico 01080, Mexico}
\email{andreas.wachtel@itam.mx}
\date{Version 12, April 18, 2018}
\thanks{This work has been partially supported by the Leverhulme Trust under grant RPG-2012-483. AW gratefully acknowledges the financial support by Asociaci\'on Mexicana de Cultura A.C.}

\begin{abstract}
  Uniform LBB conditions are of fundamental importance for the finite element solution of problems in incompressible fluid mechanics, such as the Stokes and Navier-Stokes equations. 
  In this work we prove a uniform inf-sup condition for the lowest order Taylor-Hood pairs 
  $\mathbb{Q}_2\times\mathbb{Q}_1$  and  $\mathbb{P}_2\times\mathbb{P}_1$ 
  on a family of anisotropic meshes.
  These meshes may contain  refined edge and corner patches.
  To this end, we generalise Verf\"urth's trick and recent results by some of the authors.
Numerical evidence confirming the necessity of the hypotheses on the partitions is provided.
\end{abstract}
\maketitle

\section{Introduction}

The finite element method for  the Stokes problem is subject to the satisfaction
of the discrete inf-sup condition. For an effective method, the discrete velocity and pressure spaces should
be balanced correctly. This balance results in a discrete inf-sup constant that is independent
of  mesh properties, such as the size and shape of the elements. Concerning the first requirement, there
are many finite element pairs which have been proved to be inf-sup stable on regular meshes (see \cite{BBF13} for an
extensive review). Concerning the second requirement, for the vast majority of methods the
discrete inf-sup constant has only been proved to be independent of the size of the elements, but may
depend on the aspect ratio of the elements of the partition.

This work addresses the last point raised in the previous paragraph. The stability of finite element 
pairs has been less studied in the anisotropic case, but some progress has been made for some specific
pairs, especially using discontinuous pressures and, possibly, non-conforming elements for the velocity.
For example, in the work \cite{duran2008error} the authors analyse the Raviart--Thomas element  of arbitrary order on triangular and tetrahedral meshes under a maximum angle condition.
Another  example is \cite{ANS01} where the authors analyse the Crouzeix--Raviart element in anisotropic meshes
(see also \cite{AD99}). In \cite{SS98,SSS99} the authors consider quadrilateral and triangular elements for the
$hp$-FEM, and analyse  pairs like the $\mathbb{Q}_{k+1}^2\times\mathbb{Q}_{k-1}$ ($k\ge 1$) with continuous velocities and discontinuous pressures.
Their analysis shows that this family is uniformly inf-sup stable in edge patches, but corner patches were excluded.
This is later analytically justified in \cite{AC2000}, where the quadrilateral element $\mathbb{Q}_{k+1}^2\times \mathbb{P}_{k-1}$
was analysed and it was proved that its inf-sup constant depends on the geometric properties
of the quadrilaterals in corner patches. 
In the same reference, the authors propose an enrichment of the velocity space with bubble functions whose degree depends on the aspect ratio.
Alternatively, in the recent work \cite{ABW14} a penalisation based on jumps of the pressure was designed to show inf-sup stability independent of the aspect ratio of the partition.
For a review and further references on this topic, see  \cite{Apel2003}.

In this work we study the
stability of the lowest-order Taylor-Hood pair in anisotropic meshes. This pair was originally
proposed in \cite{HT73}, and was first analysed in \cite{BP77} and \cite{Ver84} for the lowest order on triangles.
In the independent works \cite{BF91} and \cite{Ste90} the analysis was extended to higher-order families,
both in the triangular and quadrilateral cases, in two space dimensions. 
Later, in \cite{Bof97}, the three-dimensional case was addressed, using tetrahedral meshes. 

As far as we are aware, no general proof of stability  is available for the Taylor-Hood pair in the anisotropic case.
The only exception is, up to our best knowledge, the work \cite{BLR12}, where the inf-sup condition
is proved under the assumption that no drastic change of sizes of neighbouring cells in any direction may occur, see \cite[Assumption~1]{BLR12}. Moreover, for some anisotropic edge patches,
negative results, in the form of numerical experiments, are given in     \cite{SSS99} and  \cite{Apel2003}. 
In particular, in \cite{Apel2003}
instabilities are reported for the $\mathbb{P}_2^2\times\mathbb{P}_1$ pair  when the aspect ratio tends to zero in certain configurations.
The purpose of this work is then to give a rigorous proof of the inf-sup stability of the lowest order Taylor-Hood
pair, both in the quadrilateral and triangular cases. Our main result states that if the partitions have enough
internal degrees of freedom, and these are located in the right locations, then this element is uniformly
inf-sup stable. That is, the inf-sup constant is independent of the aspect ratio. Our proof will
be valid both for edge and corner patches, thus improving upon the results from \cite{BLR12}. 
Finally, we will show, by means of numerical experiments, that our assumptions are optimal, which 
complements the results from \cite{Apel2003}.

The rest of the paper is organised as follows. In \S\,\ref{Stokes-section} we present the problem of interest
and give some notation. Our main results are then stated in \S~\ref{sec:TH21} and tested numerically in 
\S\,\ref{numerics-section}. The proofs for the quadrilateral case are then presented in \S\,\ref{proof-quads},
and for the triangular case  in \S\,\ref{sec:App2x1MEP2P1}.

\section{Preliminaries and notation}\label{Stokes-section}

Let $\Omega\subset\mathbb{R}^2$ be  an open, bounded, connected and polygonal domain.
Throughout, we use standard notation for Sobolev spaces (see  \cite{GR86}), namely, for $D\subset\Omega$,  $L^2(D)$ (resp., $L^2_0(D)$) stands for the space of (generalised) functions which are square integrable in $D$ (resp., which belong to $L^2(D)$ and have zero mean value in $D$), ${H}^1(D)$ ($H^1_0(D)$) are elements of $L^2(D)$ 
whose first order derivatives belong to $L^2(D)$
(and whose trace is zero on $\partial D$).
Vector-valued spaces and functions will be denoted using  bold-faced letters.
The inner product in $L^2(D)$ (or $\b{L}^2(D)$) is denoted by $\Lprod{\cdot}{\cdot}_D$, with associated norm $\Lnorm{\cdot}{D}$,   the norm (seminorm) in $H^1(D)$ is denoted by  $\Hnorm{\cdot}{D}$ ( $\Hseminorm{\cdot}{D}$ ).
By virtue of the Poincar\'e inequality $\Hseminorm{\cdot}{D}$ is a norm on $H^1_0(D)$.
Finally, we denote by $\chi_D^{}$ the characteristic function of $D$.

A classical result (see \cite{GR86}) is the following inf-sup condition:
There exists $\beta_\Omega>0$, depending only on $\Omega$, such that
\begin{equation}
  \inf_{q\in M}\sup_{\bv\in\b{V}} \frac{\Lprod{\Div\bv}{q}_\Omega}{\Hseminorm{\bv}{\Omega}\Lnorm{q}{\Omega}} \geq \beta_\Omega >0\,,
  \label{isc:VM}
\end{equation}
where $\b{V}\times M \bydef \b{H}^1_0(\Omega)\times L^2_0(\Omega)$.
The purpose of this work is to prove the discrete analogue of this result for the lowest order Taylor--Hood element on anisotropic meshes.

\subsection{Finite element spaces}
We require a partition $\PO$ of $\Omega$ to have  the following properties.
We suppose $\Xi$ is a conforming shape regular partition of $\Omega$ into parallelograms (macro elements).
These macro element cells are denoted by $\omega$.
The partition $\PO$ is a conforming refinement of $\Xi$ and may  contain edge patches as in Figure~\ref{fig:EPs1}(d-f) 
and  corner patches as in Figure~\ref{fig:CPs2}(c-d) or as described below in Remark~\ref{rem:joinedCPs}.
An example of such a mesh is depicted in Figure~\ref{fig:Lparameterized}.
Some further refinements as in  Figure~\ref{fig:refinedPatches} are allowed.


In Section~\ref{sec:App2x1MEP2P1}, we analyse the Taylor-Hood space defined in simplicial triangulations.
For that case we will suppose that  $\PO$ is a triangulation obtained by dividing each quadrilateral of a
 partition satisfying the assumptions above into two triangles of equal area.

We now define the discontinuous space  
\[
  \mathbb{Q}_{-\ell,\PO} \bydef \set{q\in L^2(\Omega)\colon \restrict{q}{K}\in \mathbb{Q}_\ell(K) \text{ for } K\in\PO}  \spacedtext{for} \ell=1,2\,,
\]
and for  $\omega\subseteq\Omega$ the (locally) continuous spaces
\[
  \mathbb{Q}_{\ell,\PO}(\omega) \bydef \set{q\in\mathbb{Q}_{-\ell,\PO}\colon \mathrm{supp}\,q\subseteq\cl{\omega}}\cap C^0(\omega) \,.
\]
Let $\MP(\omega) \bydef \mathbb{Q}_{1,\PO}(\omega) \cap L^2_0(\omega)$; 
in the case of  $\omega = \Omega$ we write shortly $\MP$.
Let $\VP(\omega) \bydef [\mathbb{Q}_{2,\PO}(\omega)]^2\cap \b{H}^1_0(\omega)$;
in the case of  $\omega = \Omega$ we write shortly $\VP$.

Throughout this manuscript, $C$ (with or without subscript) will denote a positive constant, which will be
 independent of the size, and aspect ratio of the elements of a given partition. The value of such a constant
needs not to be equal whenever written in two different places.

\section{Main results}
\label{sec:TH21}

In this section we state the main results of this work. 
The proofs of these results are postponed to \S\,\ref{proof-quads}.
To ease the readability, we have restricted the presentation of the results to 
rectangular partitions, but the proofs can be extended
 to the case in which the whole macro element $\omega\in\Xi$ is a parallelogram, and every $K\in\PO$, $K\subseteq\omega$ is the image of the corresponding subset of the reference cell by the same affine transformation that maps $(0,1)^2$ into $\omega$.

It is worth mentioning that, under this assumption, the finite element spaces defined in the last section
are mapped finite element spaces. 
Furthermore, we sketch in  \S\,\ref{sec:App2x1MEP2P1}  how the proofs can be extended to triangulated edge and corner patches.

Below, Theorem~\ref{thm:THEPs} states the main result for edge patches,
while Theorem~\ref{thm:THCPs} covers the case of  corner patches. The results are presented for only one
patch in each case, but the proof of a global inf-sup condition (\emph{i.e.}, an inf-sup condition
involving $\VP$ and $\MP$) can be obtained using a macro-element technique, as it was done in \cite{ABW14}.

\begin{hypo}
  \label{hyp_GoodEdgePatch}
  We will suppose that every edge patch contains at least one division in the direction orthogonal to the ``long and thin'' elements.
More precisely, every ``long and thin'' element is divided into two ``long and thin'' elements.
\end{hypo}

\begin{figure}[htbp]
  \centering
	\includegraphics{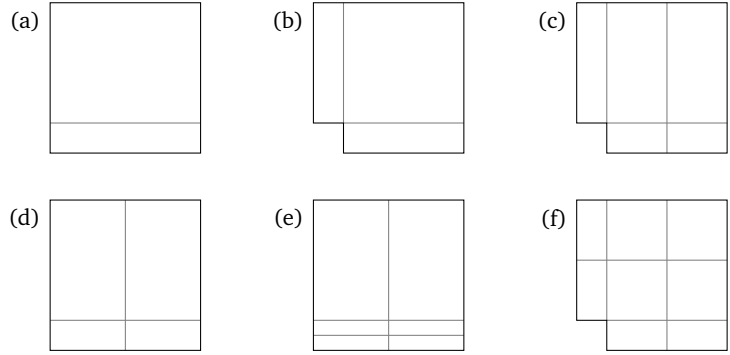}
	\caption{Edge patches. The cases in the top row 
do not satisfy Hypothesis~\ref{hyp_GoodEdgePatch}, but the cases in the bottom row do satisfy it.
The examples depicted in (b), (c) and (f)  are overlapped edge patches.}
  \label{fig:EPs1}
\end{figure}

\begin{theorem}
  \label{thm:THEPs}
Let $\omega\subset \Omega$ be partitioned as an edge patch as depicted in Figure~\ref{fig:EPs1} (d,e,f).
Then, there exists a constant $\beta>0$, independent of the shape and size of the elements in $\PO$, such that, 
 \begin{equation}\label{infsup-Q2Q1}
   \inf_{q\in \MP(\omega)}\sup_{\bv\in \VP(\omega)}\frac{\Lprod{\Div \bv}{q}_\omega}{|\bv|_{1,\omega}\Lnorm{q}{\omega}}\ge \beta\,.
\end{equation}
\end{theorem}

\begin{remark}\label{remark1}
One clear consequence of  the proof of Theorem~\ref{thm:THEPs} is that the inf-sup constant $\beta$
is the same one for the cases depicted in Figure~\ref{fig:EPs1}(d) and  Figure~\ref{fig:EPs1}(e). This opens
the door to study the case in which an edge patch has been refined further.
More precisely,  if we say that the edge patch in Figure~\ref{fig:EPs1}(d) is refined $r=0$ times, and the patch in Figure~\ref{fig:EPs1}(e) is refined $r=1$ times, then, continuing in the same way, 
\emph{i.e.}, bisecting the long and thin elements $r$ times,  we can define a patch that has been refined $r\ge 2$ times. 
Our conjecture, which is supported by numerical evidence shown in Section~\ref{numerics-section}, is that the
inf-sup constant is not affected by the value of $r$, but a formal proof of this fact is lacking.
\end{remark}

The next result concerns corner patches.
A corner patch will be decomposed as $\omega= \omega_\bc \cup \omega_E^{}$ where $\omega_\bc$ is the shape-regular small region (shaded in Figure~\ref{fig:CPs2} for each corner patch) and $\omega_E^{}$ is an ``overlapped'' edge patch as discussed in Theorem~\ref{thm:THEPs} (see Figure~\ref{fig:EPs1}).
For the  proof of Theorem~\ref{thm:THCPs}, we will need the following assumption.

\begin{hypo}
  \label{hyp_GoodCorner}
  The Partition on $\omega_\bc\subset \omega$ is such that the pair $\VP(\omega_\bc)\times\MP(\omega_\bc)$ is (uniformly) inf-sup stable with a constant $\beta_\bc$.
\end{hypo}

\begin{figure}[!hbtp]
  \centering
	\includegraphics{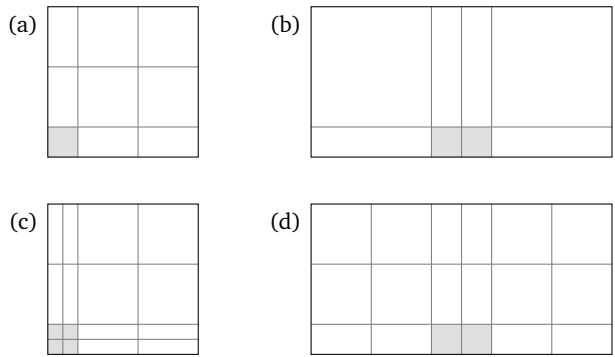}
	\caption{Corner patches. Case (a) does not satisfy Hypothesis~\ref{hyp_GoodCorner}, the others do.
	However, the assumptions of Theorem~\ref{thm:THCPs} also exclude case (b) since the ``overlapped'' edge patch does not satisfy Hypothesis~\ref{hyp_GoodEdgePatch}.}
  \label{fig:CPs2}
\end{figure}

\begin{theorem}
  \label{thm:THCPs}
  Let $\omega\subset \Omega$ be partitioned into a corner patch, that is, $\omega=\omega_\bc \cup \omega_E^{}$ with $\omega_\bc$ and $\omega_E^{}$ satisfying Hypothesis~\ref{hyp_GoodCorner} and \ref{hyp_GoodEdgePatch}, respectively (\emph{e.g.}\ as shown in Figure~\ref{fig:CPs2} (c-d)). 
 Then, there exists $\beta_\PO>0$, depending only on $\beta$ (from \eqref{infsup-Q2Q1}) and $\beta_\bc^{}$
(from Hypothesis~\ref{hyp_GoodCorner}), such that
\[
   \inf_{q\in \MP(\omega)}\sup_{\bv\in \VP(\omega)}\frac{\Lprod{\Div \bv}{q}_\omega}{|\bv|_{1,\omega}\Lnorm{q}{\omega}}\ge \beta_\PO\,.
\]
\end{theorem}

\begin{remark}\label{rem:joinedCPs}
  A comment on Hypothesis~\ref{hyp_GoodCorner} is in place.
  In the case of a single corner patch, Hypothesis~\ref{hyp_GoodCorner} requires $\omega_\bc$ to be refined (as in Figure~\ref{fig:CPs2}(c)).
  Now, if a corner patch is formed ``joining''  several patches of the type shown in Fig.2(a), see for instance Fig.2(d), then the presence of enough degrees of freedom inside of $\omega_\bc$ ensures that Hypothesis~\ref{hyp_GoodCorner} is satisfied without the need of further refinement.
\end{remark}

\section{Numerical confirmation}\label{numerics-section}
In this section we report the discrete inf-sup constants for the Taylor-Hood pairs in 
different configurations of corner and edge patches. 
Our aim is to confirm the validity of the results from last section. Then, we introduce the parameter $h$ in such a way
that the short side of an anisotropic element is of size $h$, and the long one is $O(1)$ 
(see Figure~\ref{fig:EPCSparameterized} for details). For the refinement in the direction orthogonal
to the long and thin element required by Hypothesis~\ref{hyp_GoodEdgePatch}, we have introduced an extra edge
at the midpoint of the long and thin edges (for example, in Figure~\ref{fig:EPs1}(e) the extra vertical edge
is located at $x=1/2$, and in Figure~\ref{fig:EPs1}(f) is located at $x=(1+h)/2$).

\begin{figure}[htbp]
  \centering
	\includegraphics{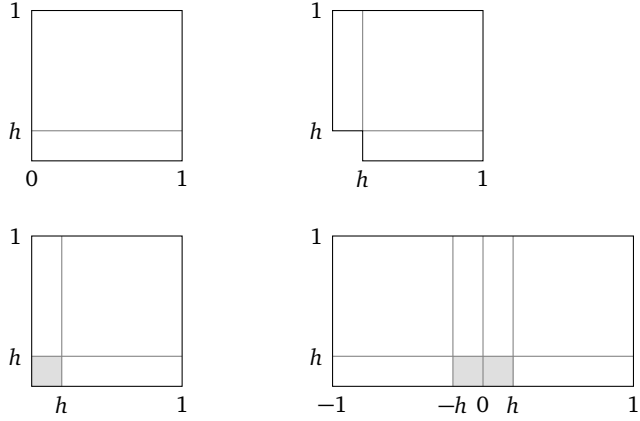}
	\caption{Parametrised geometries for the coarse partitions  of the edge and corner patches shown in Figures~\ref{fig:EPs1} and \ref{fig:CPs2}. That is,  before subdividing cells so that  Hypotheses~\ref{hyp_GoodEdgePatch} and \ref{hyp_GoodCorner} are satisfied.}
  \label{fig:EPCSparameterized}
\end{figure}

\subsection{Results on edge patches} 
Our first experiments (in Table~\ref{tab:exps-EPs}) aim at showing that Hypothesis~\ref{hyp_GoodEdgePatch} imposes  sufficient (and necessary) requirements on edge patches  such that the pair $\mathbb{Q}_2^2\times\mathbb{Q}_1$ is uniformly stable on them.
The results reported in the last two columns of Table~\ref{tab:exps-EPs} show that, once the requirements 
from Hypothesis~\ref{hyp_GoodEdgePatch} are fulfilled,
the inf-sup constant remains bounded below by a constant independent of $h$. Now, to assess the
necessity of this restriction we also report in the first two columns of Table~\ref{tab:exps-EPs} the results
for partitions that do not satisfy this hypothesis, where we see that the inf-sup constants degenerate
with $h$. 

\begin{table}[hbtp]
\centering
\caption{Discrete inf-sup constants of the pair $\mathbb{Q}_2^2\times\mathbb{Q}_1$ on edge patches shown in Figure~\ref{fig:EPs1}}
\label{tab:exps-EPs}
\begin{tabular}{lcccc}
\toprule
\multicolumn{1}{c}{$h$} 
& \multicolumn{1}{c}{Fig.\,\ref{fig:EPs1}(a)} 
& \multicolumn{1}{c}{Fig.\,\ref{fig:EPs1}(c)} 
& \multicolumn{1}{c}{Fig.\,\ref{fig:EPs1}(d)} 
& \multicolumn{1}{c}{Fig.\,\ref{fig:EPs1}(f)} 
\\
\otoprule
$10^{-1}$  & $7.12\cdot10^{-2}$  &  $1.07\cdot10^{-1}$  &  $0.426$ &    $0.487$ \\
$10^{-2}$  & $7.83\cdot10^{-3}$  &  $1.17\cdot10^{-2}$  &  $0.256$ &    $0.477$ \\
$10^{-3}$  & $7.90\cdot10^{-4}$  &  $1.20\cdot10^{-3}$  &  $0.208$ &    $0.469$ \\
$10^{-4}$  & $7.90\cdot10^{-5}$  &  $1.21\cdot10^{-4}$  &  $0.202$ &    $0.468$ \\
$10^{-5}$  & $7.91\cdot10^{-6}$  &  $1.21\cdot10^{-5}$  &  $0.201$ &    $0.468$ \\
\bottomrule
\end{tabular}
\end{table}

\subsection{Results on corner patches} 
In Table~\ref{tab:exps-CPs} we report the results obtained for different  configurations of corner patches. 
We confirm the results of Theorem~\ref{thm:THCPs} in the sense that whenever the hypotheses imposed on
the partitions are satisfied, the inf-sup constant remains bounded below by a constant
independent of $h$ (as it can be seen in the last two columns). On the contrary, the first column
of Table~\ref{tab:exps-CPs} shows that if the hypotheses are violated, then the inf-sup constant decays with $h$.

\begin{table}[hbtp]
\centering
\caption{Discrete inf-sup constants of the pair $\mathbb{Q}_2^2\times\mathbb{Q}_1$ on corner patches shown in Figure~\ref{fig:CPs2}}
\label{tab:exps-CPs}
\begin{tabular}{lccc}
\toprule
\multicolumn{1}{c}{$h$} 
& \multicolumn{1}{c}{Fig.\,\ref{fig:CPs2}(b)} 
& \multicolumn{1}{c}{Fig.\,\ref{fig:CPs2}(c)} 
& \multicolumn{1}{c}{Fig.\,\ref{fig:CPs2}(d)} 
\\
\otoprule
$10^{-1}$ &     $4.01\cdot10^{-1}$ &  $0.455$  &    $0.341$   \\ 
$10^{-2}$ &     $1.82\cdot10^{-1}$ &  $0.406$  &    $0.324$   \\ 
$10^{-3}$ &     $6.07\cdot10^{-2}$ &  $0.384$  &    $0.323$   \\ 
$10^{-4}$ &     $1.93\cdot10^{-2}$ &  $0.382$  &    $0.322$   \\ 
$10^{-5}$ &     $6.11\cdot10^{-3}$ &  $0.381$  &    $0.322$   \\ 
\bottomrule
\end{tabular}
\end{table}

\subsection{The stability on an L-shaped domain}
The previous examples were restricted to one single edge, or corner, patch. Next, we show an
example in which edge and corner patches are combined. For this, we choose 
$\Omega=int\Big([-2,1]\times[0,1] \cup [0,1]\times [-1,0]\Big)$, and partition it into an edge patch
and a corner patch, as depicted in Figure~\ref{fig:Lparameterized}. Once again the inf-sup constant stays bounded below
by a constant independent of the value of $h$.

\begin{figure}[htbp]
  \centering
  \raisebox{-.5\height}{\includegraphics{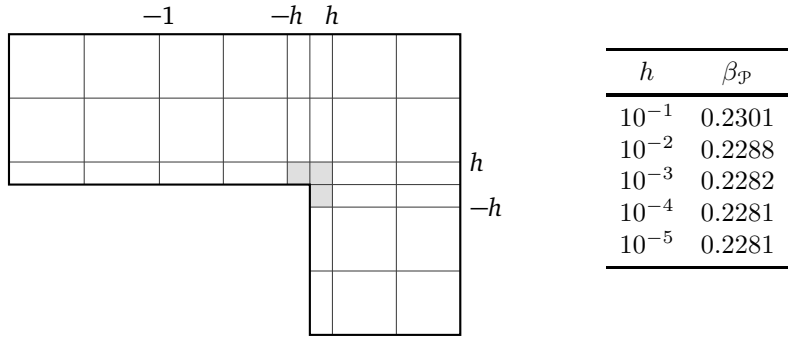}}
  \hspace{1cm}
\begin{tabular}{lc}
\toprule
\multicolumn{1}{c}{$h$} 
& \multicolumn{1}{c}{$\beta_\PO$} 
\\
\otoprule
$10^{-1}$  &    $0.2301$ \\
$10^{-2}$  &    $0.2288$ \\
$10^{-3}$  &    $0.2282$ \\
$10^{-4}$  &    $0.2281$ \\
$10^{-5}$  &    $0.2281$ \\
\bottomrule
\end{tabular}

	\caption{Discrete inf-sup constants for the Taylor-Hood element in the L-shape domain.}
  \label{fig:Lparameterized}
\end{figure}

\subsection{Refined patches}
\label{sec:refinements}

In this section we verify numerically the claims made in Remark~\ref{remark1}. For this,
we consider the refined edge and corner patches depicted in Figure~\ref{fig:refinedPatches}. 
In Table~\ref{tab:isc-refined} we report the values of the inf-sup constants. We observe that, as was conjectured
in Remark~\ref{remark1}, for the case of a refined edge patch the inf-sup constants reported
in the first two columns remain
independent of the number of refinements, thus confirming that conjecture, at least numerically.
Now, for the case of a refined corner patch, we see that, as mentioned in Remark~\ref{rem:CPrefinements} (after the proof of
Theorem~\ref{thm:THCPs}),
the value of the inf-sup constants (reported on the third column of Table~\ref{tab:isc-refined})
decreases with the number of refinements, thus confirming
the sharpness of the proof of Theorem~\ref{thm:THCPs}.

\begin{figure}[htbp]
  \centering
  \includegraphics{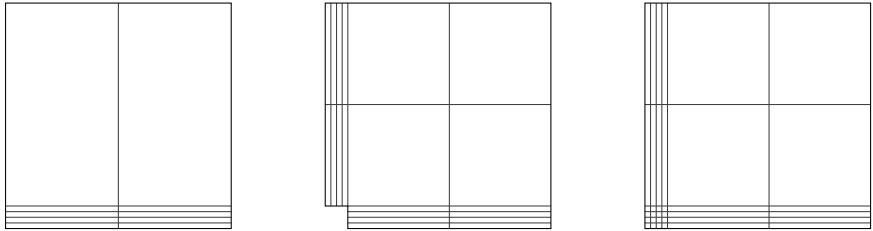}
  \caption{These patches are refinements of level $r=2$ of edge-patches and a corner patch.}
  \label{fig:refinedPatches}
\end{figure}

\begin{table}[hbtp]
\centering
\caption{We fix the parameter $h=10^{-3}$. The columns below contain the LBB constants on ($r$-times) refined  patches}
\label{tab:isc-refined}
\begin{tabular}{lcccc}
\toprule
\multicolumn{1}{c}{$r$} 
& \multicolumn{1}{c}{Fig.\,\ref{fig:refinedPatches}\,(left)}
& \multicolumn{1}{c}{Fig.\,\ref{fig:refinedPatches}\,(centre)} 
& \multicolumn{1}{c}{Fig.\,\ref{fig:refinedPatches}\,(right)} 
\\
\otoprule
$1$  &    $0.2075$ &  $0.4693$ &  $0.3843$ \\
$2$  &    $0.2075$ &  $0.4693$ &  $0.3205$ \\
$3$  &    $0.2075$ &  $0.4693$ &  $0.2503$ \\
$4$  &    $0.2075$ &  $0.4693$ &  $0.1919$ \\
\bottomrule
\end{tabular}
\end{table}

\section{Proofs on rectangular meshes}\label{proof-quads}

We start with a preliminary result. This slight generalisation of   \cite[Theorem 3.89]{VJ2016} will be the main tool we will base our
proof of stability on.

\begin{lemma}
  \label{lem:GVerfuerth}
  Let us suppose there exists a space $B\subset L^{2}(\Omega)$, such that
\begin{equation}
  \sup_{\bv\in\VP} \frac{\Lprod{q}{\Div\bv}_\Omega}{\Hseminorm{\bv}{\Omega}} \geq \beta_1 \Lnorm{q}{\Omega} 
 \spacedtext{for all} q\in B \,,
  \label{isc:assu2}
\end{equation}
and a projection $\Pi_B\colon\MP\to B$ such that
\begin{equation}
  \sup_{\bv\in\VP} \frac{\Lprod{q}{\Div\bv}_\Omega}{\Hseminorm{\bv}{\Omega}}
 \geq \beta_2 \Lnorm{q-\Pi_Bq}{\Omega} 
 \spacedtext{for all} q\in\MP \,,
  \label{isc:assu1}
\end{equation}
where $\beta_1,\beta_2$ are positive constants.
Then  $\VP\times\MP$ is inf-sup stable with inf-sup constant $\beta_0:=\frac{\beta_1\beta_2}{1+\beta_1+\beta_2}$, this
is
\[
  \sup_{\bv\in\VP} \frac{\Lprod{q}{\Div\bv}_\Omega}{\Hseminorm{\bv}{\Omega}}
 \geq \beta_0 \Lnorm{q}{\Omega} 
 \spacedtext{for all} q\in\MP \,.
\]
\end{lemma}
\begin{proof}
  Let $q\in\MP$.
  Then, using \eqref{isc:assu2} we get
  \[
	\begin{aligned}
	  \sup_{\bv\in\VP} \frac{\Lprod{q}{\Div\bv}_\Omega}{\Hseminorm{\bv}{\Omega}}
	  &= \sup_{\bv\in\VP} \set{
		\frac{\Lprod{\Pi_Bq}{\Div\bv}_\Omega}{\Hseminorm{\bv}{\Omega}}
		+
		\frac{\Lprod{q-\Pi_Bq}{\Div\bv}_\Omega}{\Hseminorm{\bv}{\Omega}}
	  }\\
	  &\geq
	  \beta_1\Lnorm{\Pi_B q}{\Omega} - \Lnorm{q-\Pi_Bq}{\Omega} \,\deeper.
	\end{aligned}
  \]
Multiplying \eqref{isc:assu1} by $(\beta_1+1)/ \beta_2$ and adding it to the last inequality  yields
\[
  \Big(1 + \frac{\beta_1+1}{\beta_2}\Big)
  \sup_{\bv\in\VP}\frac{\Lprod{q}{\Div\bv}_\Omega}{\Hseminorm{\bv}{\Omega}} \geq \beta_1\big(\Lnorm{\Pi_Bq}{\Omega} +\Lnorm{q-\Pi_B q}{\Omega} \big) \geq \beta_1\Lnorm{q}{\Omega} \,\deeper,
\]
as required.
\end{proof}

\subsection{Proof of  Theorem~\ref{thm:THEPs}}
Our aim is to prove  the assumptions of Lemma~\ref{lem:GVerfuerth}. For simplicity of the presentation
we will consider a local coordinate system, and will consider $\omega=[-H,H]\times [-h,h+2H]$ to be 
partitioned into an edge patch as shown in Figure~\ref{fig:refinedEPs}.
We will consider $\omega=M\cup M'$, where the ``bottom'' ${M= [-H,H]\times [-h,h]}$ is divided into $2$ anisotropic 
elements, as depicted in Figure~\ref{fig:EPflatpart} (left), or $4$ anisotropic elements, as in Figure~\ref{fig:EPflatpart} (right), respectively. On the other hand, the ``top'' ${M'=[-H,H]\times [h,h+2H]}$ is divided into $2$ shape-regular elements.

\begin{figure}[htbp]
  \centering
	\includegraphics{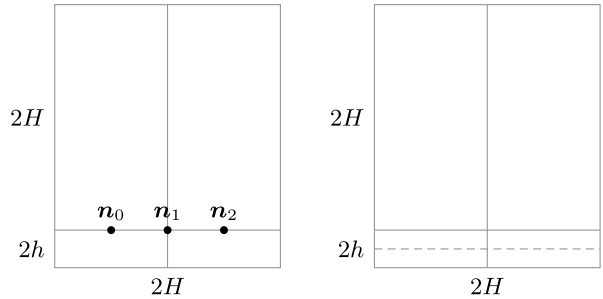}
	\caption{Edge patches in the local coordinate system.}
  \label{fig:refinedEPs}
\end{figure}

For further use we  define the following linearly independent functions 
\begin{gather}
  \phi_0 \bydef \chi_{M}^{} - \frac{\meas{M}}{\meas{M'}} \chi_{M'}^{}\,\deep, \nonumber
  \\
  \phi_{1,M}(x,y) \bydef \frac{x}{H}\chi_{M}^{} (x,y)
\spacedtext{and}
\phi_{2,M}(x,y) \bydef \Big(1 -\frac{2\abs{x}}{H} \Big)\chi_{M}^{}(x,y) \,\deep\,. \label{phi1-2}
\end{gather}

The pressure space $\MP(\omega)$  is included in $\widetilde{\MP}:=\MP(M')\oplus\MP(M)\oplus\mathrm{span}\set{\phi_0}$.
We  will prove the inf-sup condition between  this larger pressure space and $\VP(\omega)$.
As a consequence, the Taylor--Hood pair will be uniformly stable on $\omega$.
As a first step, in the next result we state a  decomposition of the  space $\MP(M)$.

\begin{figure}[bhtp]
\centering
\includegraphics{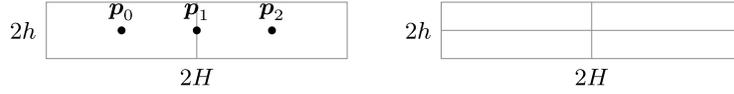}
\caption{Flat parts of edge patches.}
\label{fig:EPflatpart}
\end{figure}

\begin{lemma}
  \label{lem:decompFLATME}
  Let $M=[-H,H]\times[-h,h]$ be partitioned as in Figure~\ref{fig:EPflatpart}.
  Then, $\MP(M)$ can be decomposed as  ${\MP(M) = B_M\oplus G_M}$, where
  \[
	B_M \bydef \set{q\in\MP(M)\colon \partial_yq=0 \text{ in } M } = \mathrm{span}\set{\phi_{1,M}, \,\phi_{2,M}}\,,
\]
and
\[
G_M \bydef 
\set{q\in\MP(M)\colon \Lprod{q}{q_s}_M =0  \text{ for all } q_s\in B_M} \,.
\]
Moreover, for all $v\in H^1_0(M)$ the following holds
\begin{equation}
  \Lprod{\partial_y v}{q}_M = 0 \spacedtext{for all} q\in B_M \,,
  \label{eq:d2vbotqs}
\end{equation}
and, for every $q\in G_M$ there exists   $(0,w)\in \VP(M)\cap \b{H}^1_0(M)$ such that 
\begin{equation}
  \Lprod{\partial_y w}{q}_M = \Lnorm{q}{M}^2
  \spacedtext{and}
  \Hseminorm{w}{M} \leq C_1\Lnorm{q}{M}  \,\deeper,
  \label{eq:iscGM}
\end{equation}
where $C_1$ is independent of the size and shape of the elements in $\PO$.
\end{lemma}
\begin{proof}
As it is clear from the context, we will omit the subscript $M$ in this proof.
The  orthogonality \eqref{eq:d2vbotqs} follows using  integration by parts.
  Hence, we only need to prove \eqref{eq:iscGM}.
We split the remainder of the proof in two cases.

\noindent\underline{Case 1 :}
$M =[-H,H]\times[-h,h] $ is split at $x=0$ into two anisotropic cells belonging to $\PO$, see Figure~\ref{fig:EPflatpart}(left).
First, we extend the set $\set{\phi_{1},\phi_{2}}$ to a basis of $\MP(M) = \mathrm{span}\set{\phi_1,\phi_2,\phi_3,\phi_4,\phi_5}$ where $\phi_1,\phi_2$ are defined as  in \eqref{phi1-2}, and 
\[
  \phi_3(x,y) \bydef -\frac{y}{h}
  \;,\quad
  \phi_4 \bydef \sqrt{3}\phi_1 \phi_3
  \spacedtext{and} \phi_5 \bydef \sqrt{3}\phi_2 \phi_3 \,.
\]
A direct computation, using Fubini's Theorem and the fact that  each of the functions $\set{\phi_1,\phi_2,\phi_3}$ has zero average, yields  the orthogonalities 
\[
  \Lprod{\phi_i}{\phi_j}_M = \delta_{ij}\Lnorm{\phi_i}{M}^2=\delta_{ij}\frac{|M|}{3}\quad\textrm{for} \; i,j\in\set{1,\ldots,5}\,,
\]
and then $G_M = \mathrm{span}\,\set{\phi_3, \phi_4, \phi_5}$, as $\phi_1,\phi_2\in B_M$.
To prove \eqref{eq:iscGM}, let $g_1\in G_M$, that is,  $g_1=\sum_{i=3}^{5} q_i\phi_i$,
for some real coefficients $q_3,q_4, q_5$, 
and $\Lnorm{g_1}{M}^2 = \frac{1}{3}  \meas{M}\sum_{i=3}^{5}q_i^2$.
Let $ \boldsymbol{p}_0, \boldsymbol{p}_1,  \boldsymbol{p}_2 \in \mathring{M}$  be the locations of the degrees of freedom (dof) of $\VP(\omega)$ in $\mathring{M}$ (see Figure~\ref{fig:EPflatpart}, left),
and let $b_0, b_1, b_2\in H^1_0(M)$ be the piecewise $\mathbb{Q}_2$ bubble functions, such that $b_i(\boldsymbol{p}_j) = \delta_{ij}$ and   $(0,b_i)\in\VP(\omega)$.
Let, in addition, 
$v_3 \bydef b_2 + b_0$, $v_4 \bydef b_2-b_0$ and $v_5 \bydef b_1 -\frac{1}{4}(b_0+b_2)$.
A direct calculation using Fubini's theorem gives
\begin{equation}
  \begin{aligned}
  \Lprod{\partial_y v_3}{\phi_3}_M &
  = \frac{16H}{9}
  \,, \\[1ex]
  \Lprod{\partial_y v_3}{\phi_4}_M &=0  \,, \\[1ex]
  \Lprod{\partial_y v_3}{\phi_5}_M &=0  \,,
  \end{aligned}
  \qquad
  \begin{aligned}
	\Lprod{\partial_y v_4}{\phi_3}_M &=0  \,, \\[1ex]
  \Lprod{\partial_y v_4}{\phi_4}_M &
  = \frac{8H}{3\sqrt{3}}   
  \,, \\[1ex]
  \Lprod{\partial_y v_4}{\phi_5}_M &=0  \,,
  \end{aligned}
  \qquad
  \begin{aligned}
	\Lprod{\partial_y v_5}{\phi_3}_M &=0  \,, \\[1ex]
	\Lprod{\partial_y v_5}{\phi_4}_M &=0  \,, \\[1ex]
	\Lprod{\partial_y v_5}{\phi_5}_M &= \frac{4H}{3\sqrt{3}}\,.
  \end{aligned}
  \label{eq:dyv2x1ME}
\end{equation}
Hence,  defining $v^\star = \alpha_3 q_3 v_3 + \alpha_4 q_4 v_4 + \alpha_5 q_5 v_5$ 
  with $\alpha_3 = \frac{3h}{4}, \alpha_4 =\frac{\sqrt{3} h}{2}$ and $\alpha_5 =\sqrt{3}h$, 
we obtain
\begin{equation}\label{star-star}
  \Lprod{\partial_y v^\star}{g_1}_M = \sum_{i=3}^{5}\Lnorm{q_i\phi_i}{M}^2 = \Lnorm{g_1}{M}^2 \,\deeper.
\end{equation}
Finally,  since $\abs{\alpha_i} \leq \sqrt{3} h$ and
$\Hseminorm{v_i}{M}^2\leq  \tilde{C} h^{-2}\meas{M}/3$,  we get
\begin{equation}\label{square}
  \Hseminorm{v^\star}{M}^2
  \leq 3\sum_{i=3}^{5}\Hseminorm{\alpha_i q_iv_i}{M}^2
  \leq 3\tilde{C} \sum_{i=3}^{5} q_i^2 h^2 \left( h^{-2}\meas{M} \right)
  = 9\tilde{C} 
        \sum_{i=3}^{5} \Lnorm{q_i\phi_i}{M}^2
		\,\deeper,
\end{equation}
which, since  $(0, v^\star)\in\VP(M)$, finishes the first case with $C_1=3\sqrt{\tilde{C}}$.

\noindent\underline{Case 2 :}
We now consider the case where $M\bydef[-H,H]\times[-h,h]$ 
is divided into four rectangles $K\in\PO$ obtained by splitting $M$ along the lines $x=0$ and $y=0$,
as depicted in Figure~\ref{fig:EPflatpart} (right).
We extend the set $\set{\phi_1,\phi_2, \phi_3,\phi_4,\phi_5}$, introduced in the proof of the first case, 
to a basis of $\MP(M)$ by adding  the following functions
\[
  \phi_6(x,y) \bydef  \begin{cases}
	\phi_3 (x,2y-h) &\text{, if } y>0\,,\\
	-\phi_3 (x,2y+h) &\text{, if } y\leq0\,.
  \end{cases}
  \;,\quad
  \phi_7 \bydef \sqrt{3}\phi_1 \phi_6
  \;,\;
  \phi_8 \bydef \sqrt{3}\phi_2 \phi_6 \,.
\]
These functions are even in $y$, and then, proceeding as before we get
\[
  \Lprod{\phi_i}{\phi_j}_M = \delta_{ij}\Lnorm{\phi_i}{M}^2 = \delta_{ij}\frac{\meas{M}}{3} \spacedtext{for} i,j\in\set{1,\ldots,8}.
\]

Let now $\tilde{g}=g_1+g_2\in G_M$, where $g_1=\sum_{i=3}^5q_i\phi_i$, and 
 $g_2= \sum_{i=6}^{8}q_i\phi_i$, for some real coefficients $q_3,\ldots, q_8$. Thanks to the definition of the function $\phi_6$, we observe
that $g_2$ can be expressed as follows
\[
  g_2(x,y)= 
  \begin{cases}
	(q_6\phi_3 +q_7\phi_4 +q_8\phi_5)(x,2y-h) &\text{, if } y>0\,,\\
	-(q_6\phi_3 +q_7\phi_4 +q_8\phi_5)(x,2y+h) &\text{, if } y\leq0\,.
  \end{cases}
\]
Then, using the function $v^\star$ from the previous case, we consider the function $\tilde{v}:=v^\star+v_2^\star$, where
\[
  v_2^\star (x,y)\bydef
  \begin{cases}
	\left(\alpha_6q_6v_3+\alpha_7q_7v_4+\alpha_8q_8v_5\right)(x,2y-h) &\text{, if }y>0\,,\\
	-\left(\alpha_6q_6v_3+\alpha_7q_7v_4+\alpha_8q_8v_5\right)(x,2y+h) &\text{, if }y\leq0\,,\\
  \end{cases}
\]
where $\alpha_i, i=6,7,8$ are geometric constants to be determined.
We note that $(0,v_2^\star)\in\VP(M)$.

Now, we prove $\Lprod{\partial_y (v^\star + v_2^\star)}{g_1 + g_2}_M = \Lnorm{\tilde{g}}{M}^2$. 
To this end, we recall \eqref{star-star} 
and since $g_2$ and $\partial_y v_2^\star$ are even in $y$, and $g_1$ and $\partial_y v^\star$ are odd in $y$, we obtain the orthogonalities
\[
  \Lprod{\partial_y v_2^\star}{g_1}_M =0
  \spacedtext{and}
  \Lprod{\partial_y v^\star}{g_2}_M =0
  \,.
\]
In addition,  integrating by substitution and applying \eqref{eq:dyv2x1ME} gives
\[
\Lprod{\partial_y v_2^\star}{g_2}_M 
= 2\int_0^h\int_{-H}^H \partial_y v_2^\star g_2 \,\mathrm{d}x\,\mathrm{d}y
= 2\sum_{i=6}^{8} q_i^2\, \alpha_i\Lprod{\partial_y v_{i-3}}{\phi_{i-3}}_{M} \,\deeper,
\]
which on choosing 
$\alpha_i \bydef \alpha_{i-3}/2$\, ($i=6,7,8$) 
gives
\[
  \Lprod{\partial_y (v^\star + v_2^\star)}{\tilde{g}}_M = 
  \sum_{i=3}^{5}\Lnorm{q_i\phi_{i}}{M}^2 + \sum_{i=6}^{8}\Lnorm{q_i\phi_{i-3}}{M}^2 =\Lnorm{\tilde{g}}{M}^2 \,\deeper.
\]

Finally, we prove $\Hseminorm{v}{M}\leq C_1\Lnorm{q}{M}$.
First, applying Cauchy's inequality and considering the scaling (w.r.t.\ $y$) inside $v_2^\star$ as well as $\alpha_i = \alpha_{i-3}/2$, we get
\[
  \Hseminorm{v_2^\star}{M}^2
  \leq 3 \bigg( 4 \sum_{i=6}^8 \Hseminorm{\alpha_iq_iv_{i-3}}{M}^2 \bigg)
  = 3 \bigg(  \sum_{i=6}^8 q_i^2\Hseminorm{\alpha_{i-3}v_{i-3}}{M}^2 \bigg) \,.
\]
Now, using $\Hseminorm{\alpha_iv_i}{M}^2\leq \tilde{C}\meas{M} = 3\tilde{C}\Lnorm{\phi_i}{M}^2$ (as in Case~1) we get
\[
\Hseminorm{v^\star_2}{M}^2\leq  3\tilde{C}  \sum_{i=6}^{8}\Lnorm{q_i\phi_{i-3}}{M}^2 
= 9\tilde{C} \Lnorm{g_2}{M}^2\,,
\]
where $\tilde{C}$ is the same constant as in \eqref{square}.
Finally, recalling that $v^\star$ is even in $y$, and
$v_2^\star$ is odd in $y$, we get $\Lprod{\nabla v^\star}{\nabla v_2^\star}_M=0$ and arrive at
\begin{equation}\label{18+1}
  \begin{aligned}
\Hseminorm{v^\star+v_2^\star}{M}^2 
&= \Hseminorm{v^\star}{M}^2 + \Hseminorm{v_2^\star}{M}^2
\\
&\leq
 9\tilde{C} \big(\Lnorm{g_1}{M}^2 + \Lnorm{g_2}{M}^2\big)
= 9\tilde{C} \Lnorm{\tilde{g}}{M}^2 \,\deeper,
  \end{aligned}
\end{equation}
which finishes the proof in this case, again with the same constant $C_1=3\sqrt{\tilde{C}}$.
\end{proof}

Now we rewrite $\widetilde{\MP} =\MP(M')\oplus\MP(M)\oplus\mathrm{span}\set{\phi_0} = G_\omega\oplus B_\omega$
where
\begin{equation}
  G_\omega \bydef \MP(M')\oplus G_M
  \spacedtext{and}
  B_\omega \bydef B_M \oplus \mathrm{span}\set{\phi_0} \,,
  \label{def:Bomega}
\end{equation}
with  $G_M$ and $B_M$  defined as in Lemma~\ref{lem:decompFLATME}.
Let 
   \begin{equation}\label{17-2}
	 \VX \bydef \VP(M')\oplus\set{(0,v)\in\VP(M)} \subset \VP(\omega)\,.
\end{equation}
	Then, using  Lemma~\ref{lem:decompFLATME} and the stability of the Taylor--Hood pair on $M'$
(which is due to $M'$ being partitioned in a shape regular way),
we conclude the uniform inf-sup condition
\begin{equation}
  \inf_{q\in G_\omega}\sup_{\bv\in\VX} \frac{\Lprod{\Div\bv}{q}_\omega}{\Lnorm{q}{\omega}\Hseminorm{\bv}{\omega}} 
  \geq \beta_2 \,,
  \label{isc:VPlocGomega}
\end{equation}
where $\beta_2$ is independent of the aspect ratio, \emph{i.e.}\ the quotient $h/H$.
From \eqref{isc:VPlocGomega} it follows that $G_\omega$ is controlled by a subset of $\VP(\omega)$ that vanishes on $\gamma:=M\cap M'$. To
control the remaining part of $\widetilde{\MP}$ we need the bubble functions connecting $M$ and $M'$. 
This is stated in the next result.

\begin{lemma}
  \label{lem:uiscBOmega2}
 Let $B_\omega$ be defined as in \eqref{def:Bomega}.
  Then, there exists a constant $\beta_1>0$ independent of $h$ and $H$, such that
  \[
	\sup_{\bv\in\VP(\omega)} \frac{\Lprod{\Div\bv}{q}_\omega}{\Hseminorm{\bv}{\omega}} \geq \beta_1 \Lnorm{q}{\omega} 
	\spacedtext{for all} q\in B_\omega \,.
  \]
\end{lemma}
\begin{proof}
  Let $\boldsymbol{n}_0, \boldsymbol{n}_1, \boldsymbol{n}_2 \in\gamma$ be the locations of the degrees of freedom (dof) of $\VP(\omega)$ on $\gamma$ (see Figure~\ref{fig:refinedEPs}),
  and let $f_0, f_1, f_2\in H^1_0(\omega)$ be the unique piecewise $\mathbb{Q}_2$ functions 
such that $f_i(\boldsymbol{n}_j) = \delta_{ij}$,  $(0,f_i)\in\VP(\omega)$, and $f_i$ are piecewise affine in $y$.
  Let $q\in B_\omega$ be arbitrary, that is,  $q = q_0 \phi_0 + q_1\phi_{1,M} + q_2\phi_{2,M}$ for some $q_0,q_1,q_2\in\RR$.
  Let $w^\star \bydef  \sum_{i=0}^{2}\alpha_iw_i$ where
  $\alpha_i$ are coefficients  to be chosen and 
  $w_0 \bydef f_0 + f_2$, $w_1 \bydef f_2-f_0$ and $w_2 \bydef f_1 -\frac{1}{4}(f_0+f_2)$. 
  Using the fact that  $\gamma$ consists of two edges of equal lengths, a direct computation gives the orthogonalities 
\[
  \begin{aligned}
  \Lprod{w_0}{\phi_{1,M}}_\gamma &=0  \,, \\ 
  \Lprod{w_0}{\phi_{2,M}}_\gamma &=0  \,,
  \end{aligned}
  \qquad
  \begin{aligned}
  \Lprod{w_1}{\jump{\phi_0}}_\gamma &=0  \,, \\ 
  \Lprod{w_1}{\phi_{2,M}}_\gamma &=0  \,,
  \end{aligned}
  \qquad
  \begin{aligned}
	\Lprod{w_2}{\jump{\phi_0}}_\gamma &=0  \,, \\ 
  \Lprod{w_2}{\phi_{1,M}}_\gamma &=0  \,.
  \end{aligned}
\]
 Then, integration by parts  gives
\[
  \begin{aligned}
  \Lprod{\partial_y w^\star}{q}_{\omega} 
  &= \Lprod{w^\star }{\jump{q}}_{\gamma}
  = \alpha_0q_0\Lprod{w_0}{\jump{\phi_0}}_{\gamma} 
  + \sum_{i=1}^{2} \alpha_iq_i\Lprod{w_i}{\phi_{i,M}}_\gamma \,\deeper.
  \end{aligned}
\]
In addition
\[
  \begin{aligned}
	\Lprod{w_1}{\phi_{1,M}}_{\gamma}
	&= \frac{2H}{3} \,,
	\quad
	\Lprod{w_2}{\phi_{2,M}}_{\gamma}
	= \Lprod{f_1}{\phi_{2,M}}_{\gamma}
	%
	= \frac{H}{3}
	\,, \quad\text{and}
	\\
  \Lprod{w_0}{\jump{\phi_0}}_{\gamma}  
  &= \frac{4H}{3}\jump{\phi_0}_\gamma
  \,\deep.
  \end{aligned}
\]
Since $\meas{M}\jump{\phi_0} = \Lnorm{\phi_0}{\omega}^2$
and $\Lnorm{\phi_{i,M}}{\omega}^2 = \Lnorm{\phi_i}{M}^2 = (2H/3)2h$,
we set 
$\alpha_0 \bydef  3h q_0$
$\alpha_1 = 2h q_1$
and
$\alpha_2 = 4h q_2$
to obtain
\[
  \Lprod{\partial_y w^\star}{q}_{\omega}  = \Lnorm{q}{\omega}^2 \,\deeper.
\]
Finally, since $\abs{\alpha_i} \leq 4hq_i$ and
$\Hseminorm{w_i}{\omega}^2\leq 3 \sum_{j=0}^{2}\Hseminorm{f_j}{\omega}^2 
\leq C\left( h^{-2}\meas{M} + H^{-2}\meas{M'}\right)$,  we get
\begin{multline*}
  \Hseminorm{w^\star}{\omega}^2
  \leq 3 \sum_{i=0}^{2}\Hseminorm{\alpha_i w_i}{\omega}^2
  \leq C \sum_{i=0}^{2} q_i^2 h^2 \bigg( \frac{1}{h^2}\meas{M} + \frac{1}{H^2}\meas{M'}\bigg)
  \\
  \leq C \sum_{i=0}^{2} q_i^2 \meas{M} \bigg( 1 + \frac{\meas{M}}{\meas{M'}}\bigg)
  \leq C \sum_{i=0}^{2} \Lnorm{q_i\phi_i}{\omega}^2
  =: \beta_1^{-2}  \Lnorm{q}{\omega}^2 \,\deeper,
\end{multline*}
which finishes the proof since  $(0, w^\star)\in\VP(\omega)$.
\end{proof}

The proof of Theorem~\ref{thm:THEPs} appears then as an application of Lemma~\ref{lem:GVerfuerth}.
In fact, recall the decomposition  $\MP(\omega) \subset \widetilde{\MP}=G_\omega \oplus B_\omega$,
  where $G_\omega, B_\omega$ are defined in \eqref{def:Bomega}.
Let $q\in G_\omega\oplus B_\omega$,
and recall the definition \eqref{17-2} of  $\VX$.
  Let also $\Pi_Bq$ be the $L^2(\omega)$ projection of $q$ onto $B_\omega$.
  Then, from \eqref{eq:d2vbotqs} 
  and the fact that $\Pi_B(q)$ is  constant in  $M'$, we  get
  $\Lprod{\Pi_Bq}{\Div\bv}_\omega =0$ for all $\bv\in\VX$.
  Hence
\[
  \begin{aligned}
  \sup_{\bv\in\VP} \frac{\Lprod{q}{\Div\bv}_\omega}{\Hseminorm{\bv}{\omega}}
  &\geq  \sup_{\bv\in\VX} \frac{\Lprod{q}{\Div\bv}_\omega}{\Hseminorm{\bv}{\omega}}
  \\
  &=  \sup_{\bv\in\VX} \frac{\Lprod{q-\Pi_Bq}{\Div\bv}_\omega}{\Hseminorm{\bv}{\omega}}
  \geq \beta_2 \Lnorm{q-\Pi_Bq}{\omega}\,\deeper,
  \end{aligned}
\]
where we applied $q-\Pi_Bq\in G_\omega$ and \eqref{isc:VPlocGomega}. 
Finally Lemma~\ref{lem:uiscBOmega2} gives
\[
  \sup_{\bv\in\VP} \frac{\Lprod{\Pi_Bq}{\Div\bv}_\omega}{\Hseminorm{\bv}{\omega}}
  \geq \beta_1 \Lnorm{\Pi_Bq}{\omega} \,\deeper,
\]
which finishes the proof upon application of Lemma~\ref{lem:GVerfuerth} with a constant $\beta = \frac{\beta_1\beta_2}{1+\beta_1+\beta_2}$ independent of size and aspect ratios of the edge patch.

\begin{remark} 
  The essential part of the independence of the constant $\beta$ (in Theorem~\ref{thm:THEPs}) of the aspect ratio is given by
the proof of Lemma~\ref{lem:decompFLATME}. We observed in that proof that, for both cases, the constant 
$C_1$ is given by $3\sqrt{\tilde{C}}$, and then, it was unaffected by the additional refinement of the
edge patch (see the inequalities \eqref{square} and \eqref{18+1}). 
This confirms what was claimed in Remark~\ref{remark1}, and is affirmed by numerical experiments.
\end{remark}

\begin{remark}
  It is worth noticing that the proof of  Theorem~\ref{thm:THEPs} provides the inf-sup stability of a family of elements which contains the lowest order Taylor--Hood pair.
  In fact, the pressure space is allowed to be discontinuous across $\gamma$, which somehow generalises the results known so far.
\end{remark}

\begin{remark}
  A closer  look at the proofs of Lemmas~\ref{lem:decompFLATME}, \ref{lem:uiscBOmega2} and Theorem~\ref{thm:THEPs}
  shows  that the technique can be applied to ``overlapped'' edge patches, as shown for instance in Figure~\ref{fig:EPs1}(f) or Figure~\ref{fig:CPdecomposed}\,(centre).
  In this case, Lemma~\ref{lem:decompFLATME} has to be applied twice, that is, 
 there are two flat parts, say $M_1$ and $M_2$, where $M_2$ contains the vertically aligned ``long and thin'' cells.
 Then, $\MP(M_2)$ has to be decomposed into 
  $B_{M_2} = \set{q\in\MP(M_2)\colon \partial_xq=0 \text{ in }M_2}$
  and
  $G_{M_2}$.
  Hence, using the same arguments there exists $(v,0)\in\VP(M_2)$ such that 
  $\Lprod{\partial_x v}{g}_{M_2} = \Lnorm{g}{M_2}^2$
  and 
  $\Hseminorm{v}{M_2}\leq C_1\Lnorm{g}{M_2}$ for all $g\in G_{M_2}$.
  Then, similar to \eqref{def:Bomega} and the arguments thereafter, the pair consisting of the velocity space 
  $\VX \bydef \set{(0,v)\in\VP(M_1)}\oplus\VP(M')\oplus\set{(v,0)\in\VP(M_2)}$
  and the pressure space 
  $G_\omega \bydef G_{M_1}\oplus \MP(M') \oplus G_{M_2}$
  is uniformly inf-sup stable.
  Then, an adapted proof of Lemma~\ref{lem:uiscBOmega2} shows 
  that $B_\omega = B_{M_1}\oplus B_{M_2}\oplus \mathrm{span}\set{\phi_0, \phi_0'}$, where $\phi_0'$ is another piecewise constant function, is controlled by the six bubbles on the edges connecting $M_1, M'$ and $M_2, M'$.
  Finally, the proof of Theorem~\ref{thm:THEPs} follows analogously considering these modified definitions of $G_\omega, B_\omega$  and $\VX$.
  
\end{remark}

\subsection{Proof of  Theorem~\ref{thm:THCPs}}

\begin{figure}[bhtp]
  \centering
	\includegraphics{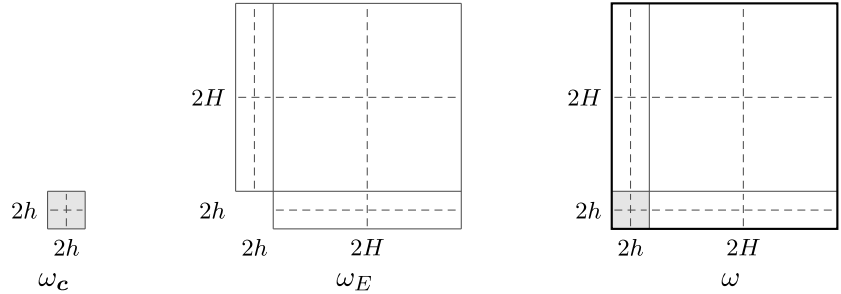}
	\caption{A corner patch  decomposed into $\omega_\bc$ and $\omega_E^{}$ satisfying Hypothesis~\ref{hyp_GoodCorner} and \ref{hyp_GoodEdgePatch}, respectively.}
  \label{fig:CPdecomposed}
\end{figure}

In this section we use notations $\omega_\bc, \omega_E^{}, \omega$ as suggested in Figure~\ref{fig:CPdecomposed} for the complete and parts of the (refined) corner patch.

As done in the previous section we see that
\[
\MP(\omega) \subset \widetilde{\MP}:=\MP(\omega_E)\oplus\MP(\omega_\bc) \oplus\mathrm{span}\set{\phi_\bc} \,,
\]
where $\phi_\bc$ is defined as
\[
\phi_\bc(x,y)=\left\{\begin{array}{ll} 1 & \textrm{if}\; (x,y)\in \omega_\bc\,,\\ -\frac{|\omega_\bc|}{|\omega_E^{}|}
& \textrm{otherwise}\,. \end{array}\right.
\]

We apply the technique  developed in \cite{ABW14} to  prove Theorem~\ref{thm:THCPs}.
Thanks to Theorem~\ref{thm:THEPs}, we know that the pair $\VP(\omega_E^{})\times\MP(\omega_E^{})$ is uniformly inf-sup stable, that is, its stability constant does not depend on $Hh^{-1}$. 
In addition, if we recall the Hypothesis~\ref{hyp_GoodCorner}, then the Taylor--Hood pair is inf-sup stable on $\omega_\bc$.

Let now  $q\in\MP(\omega)$, that is,  $q = q^\star + \Pi_\bc q$, where $q^\star \in\MP(\omega_\bc)\oplus\MP(\omega_E)$ and $\Pi_\bc q\in\mathrm{span}\set{\phi_\bc}$ is constant on $\omega_\bc$ and $\omega_E$.
Hence, the inf-sup conditions on $\omega_\bc$ and $\omega_E^{}$ (mentioned above) imply
   \[
	 \sup_{\bv\in\VP(\omega)} \frac{\Lprod{q}{\Div\bv}_\omega}{\Hseminorm{\bv}{\omega}} 
	 \geq\sup_{\bv\in\VP(\omega_E)\oplus\VP(\omega_\bc)} \frac{\Lprod{q^\star}{\Div\bv}_\omega}{\Hseminorm{\bv}{\omega}}
	 \geq \min\set{\beta, \beta_\bc}\Lnorm{q^\star}{\omega}\,,
   \]
   where $\beta$ is from Theorem~\ref{thm:THEPs} and $\beta_\bc$ is from Hypothesis~\ref{hyp_GoodCorner}.
   We will then finish the proof by showing there exists a constant $C>0$ such that
   \[
   \Lnorm{q^\star}{\omega} \geq  C \Lnorm{\Pi_\bc q}{\omega} \,\deeper.\]
   To this end, we note that the projection $\Pi_\bc q$ is given by
   \[
	 \Pi_\bc q= \avge{q}{\omega_\bc}\chi_{\omega_\bc}^{} + \avge{q}{\omega_E}\chi_{\omega_E}^{} \,\deep,
   \]
   and since $0=\avge{q}{\omega}$, we have $\avge{q}{\omega_E} = -\meas{\omega_\bc}\meas{\omega_E}^{-1} \avge{q}{\omega_\bc}$.
   Hence
   \begin{equation}
	 \Lnorm{\Pi_\bc q}{\omega}^2
	 = \meas{\omega_E}\avge{q}{\omega_E}^2 + \meas{\omega_\bc}\avge{q}{\omega_\bc}^2 
   = \left(\frac{\meas{\omega_\bc}}{\meas{\omega_E}} + 1\right)  \meas{\omega_\bc}\avge{q}{\omega_\bc}^2 
   = C_3 \meas{\omega_\bc}\avge{q}{\omega_\bc}^2 \,,
	 \label{eq:L2EqualAVG}
   \end{equation}
   where $C_3 = \meas{\omega}\meas{\omega_E}^{-1}\ge 1$.
   Now let $\Gamma_\bc\bydef \omega_\bc\cap \omega_E$, and let $e\subset\Gamma_\bc$ be an arbitrary edge satisfying $e= K_e\cap K_e'$ with $K_e\subset\omega_\bc$.
   Then, from  $\avge{q}{\omega_\bc} - \avge{q}{\omega_E} = \avge{\jump{\Pi_\bc q}}{e}$,
   and since $q$ is continuous in $\omega$,
   we get 
   \[
	 C_3  \avge{q}{\omega_\bc} = \avge{\jump{\Pi_\bc q}}{e} = -\avge{\jump{q^\star}}{e}
	  \spacedtext{and}
	  \avge{q}{\omega_\bc}^2 = C_3^{-2} \avge{\jump{q^\star}}{e}^2 \,\deeper.
   \]
   Then,  Cauchy's inequality and \cite[Lemma 2.1]{ABW14} yield
   \[
	 \begin{aligned}
	 \avge{q}{\omega_\bc}^2
	 &=   \frac{1}{\meas{\Gamma_\bc}} \int_{\Gamma_\bc}  \avge{q}{\omega_\bc}^2
	 \\
	 &=   \frac{1}{C_3^2\meas{\Gamma_\bc}}\sum_{e\subset\Gamma_\bc}\int_{e} \avge{\jump{q^\star}}{e}^2
	 \\
	 &\leq   \frac{1}{C_3^2\meas{\Gamma_\bc}}\sum_{e\subset\Gamma_\bc} \int_{e} \frac{1}{\meas{e}} \Lnorm{\jump{q^\star}}{e}^2
	 \\
	 &\leq   \frac{1}{C_3^2}\sum_{e\subset\Gamma_\bc} \frac{\meas{e}}{\meas{\Gamma_\bc}}  \left( \frac{1}{\meas{K_e}} + \frac{1}{\meas{K_e'}}\right) 4 \Lnorm{q^\star}{K_e\cup K_e'}^2 \,\deeper.
	 \end{aligned}
   \]
   Now, using $\meas{K_e}\leq \meas{K_e'}$,
   $\frac{\meas{e}}{\meas{\Gamma_\bc}} \frac{\meas{\omega_\bc}}{\meas{K_e}}=1$ 
   and
   $1/C_3\leq1$
   we get from \eqref{eq:L2EqualAVG}
  \[
	 \Lnorm{\Pi_\bc q}{\omega}^2
	 = C_3 \meas{\omega_\bc}\avge{q}{\omega_\bc}^2
	 \leq \frac{1}{C_3}  \sum_{e\subset\Gamma_\bc} \frac{\meas{e}}{\meas{\Gamma_\bc}}\frac{\meas{\omega_\bc}}{\meas{K_e}} 8 \Lnorm{q^\star}{K_e\cup K_e'}^2
	 \leq 8\Lnorm{q^\star}{\omega}^2 \,\deeper,
   \]
   as required.

\begin{remark}\label{rem:CPrefinements}
In this proof we have not considered the possibility of refining the corner patches. 
If we were to consider that case, then in the last step of the above proof we would be led to
use the fact that, for a corner patch that has been refined uniformly $r$ times, we have
\[
\frac{\meas{e}}{\meas{\Gamma_\bc}} \frac{\meas{\omega_\bc}}{\meas{K_e}}=2^{r-1}\,.
\]
This would give as a result the inf-sup constant 
\[
\beta_\PO^{} \sim C 2^{-\frac{r-1}{2}}
\]
which shows a dependency of the inf-sup constants on the refinement of the partition.
This dependency is not an artefact of the proof, as the numerical results in Section~\ref{sec:refinements} show.
\end{remark}

\section{The triangulated case}
\label{sec:App2x1MEP2P1}

In this section we prove  Theorem~\ref{thm:THEPs} and \ref{thm:THCPs} for triangulated edge- and corner patches.
A single change is required to prove Theorem~\ref{thm:THCPs}, \emph{i.e.}, we have to replace the used discrete trace estimate by one that is valid on anisotropic triangles, see \emph{e.g.}\ \cite[Theorem 3]{WH03}.

Then it remains to prove Theorem~\ref{thm:THEPs}.
To this end, we note that the proofs of Lemma~\ref{lem:GVerfuerth}, Lemma~\ref{lem:uiscBOmega2} remain valid
and we only need  to prove  Lemma~\ref{lem:decompFLATME} for triangulated anisotropic (flat) macro elements.
\subsection{Proof of Theorem~\ref{thm:THEPs}}
We perform a direct computation which  is simplified by using  the following three-point quadrature formula 
 \[
   \int_{K} f \,\mathrm{d}\b{x} \approx \frac{\meas{K}}{3}\sum_{i=1}^{3} f( \b{m}_i) \,,
 \]
 where $\b{m}_i$ are the mid-points of the edges of triangle~$K$.
 This formula is exact for $f\in\mathbb{P}_2(K)$.

 Now, we let $M =[-H,H]\times[-h,h]$ and notice that however we triangulate $M$,
 the space 
	$B_M = \mathrm{span}\set{\phi_{1}, \,\phi_{2}}$
 defined in Lemma~\ref{lem:decompFLATME} is a subspace of  $\MP(M)$.
Next we  prove \eqref{eq:iscGM} on the triangulated macro element.

 \underline{Case 1:} $M$ is split at $x=0$ and then triangulated.
 Again, we have $\dim G_M =3$ and 
 for a basis of $G_M$ we choose $\set{\varphi_i}_{i=3}^5$ to be the nodal interpolants of $\set{\phi_i}_{i=3}^5$ in Lemma~\ref{lem:decompFLATME}.
 These functions take the values $+1/-1/0$ at the signs $+/-/0$ shown in Figure~\ref{fig:pressurePatterns}.
 For simplicity (but without loss of generality) we chose a triangulation on which $\varphi_3,\varphi_5$ remain even in $x$, while $\varphi_4$ is odd in $x$. Additionally the product $\varphi_3\varphi_5$ vanishes in each of the quadrature points, and therefore $\set{\varphi_i}_i$ is an orthogonal basis of $G_M$.
 Hence,  for 
 $g\in G_M$ with  $g=\sum_{i=3}^{5}q_i\varphi_i$,  the quadrature formula yields
\[
  \Lnorm{g}{M}^2 = \sum_{i=3}^{5}\Lnorm{q_i\varphi_i}{M}^2
  = \frac{\meas{M}}{6} \left( 2q_3^2 + q_4^2 + 2q_5^2 \right) \,.
\]

\begin{figure}[htbp]
   \centering
   \begin{tabular}{ccc}
\includegraphics{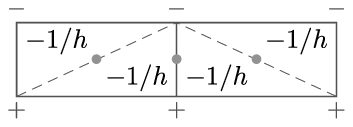}&
\includegraphics{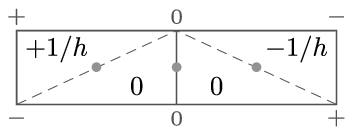}&
\includegraphics{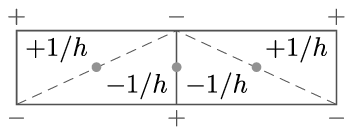} \\
$\varphi_3$ & 
$\varphi_4$ & 
$\varphi_5$
   \end{tabular}
   \caption{Basis of $G_M$ with sign-patterns of $\varphi_i$ and $\partial_y\varphi_i$.}
   \label{fig:pressurePatterns}
\end{figure}

Now, let $b_0,b_1,b_2$ be the quadratic bubble functions, each supported on two triangles and taking the value $1$ at the midpoint of the shared edge.
Then defining 
$v_3,v_4,v_5$ similar to   Lemma~\ref{lem:decompFLATME} by
$v_3 \bydef b_2+b_0$, $v_4  \bydef b_2-b_0$ and $v_5  \bydef 2 b_1 - (b_2+b_0)$
we get
\[
  \Lprod{\partial_yv_i}{g}_M 
  =-\Lprod{v_i}{\partial_yg}_M 
  =-q_i\Lprod{v_i}{\partial_y\varphi_i}_M \spacedtext{for} i=3,4,5\,,
\]
and then $v^\star \bydef \sum_{i=3}^{5} \alpha_iv_i$ satisfies
\[
  \Lprod{\partial_yv^\star}{g}_M 
  =-\Lprod{v^\star}{\partial_yg}_M 
  =-\sum_{i=3}^{5}\alpha_iq_i\Lprod{v_i}{\partial_y\varphi_i}_M \,\deeper.
\]
To calculate these products we use the quadrature formula and the values of  $\restrict{(\partial_y\varphi_i)}{K}$, given inside the triangles shown in Figure~\ref{fig:pressurePatterns}.
We obtain
\begin{gather*}
  \Lprod{v_3}{\partial_y\varphi_3}_M = -\meas{M}/(3h) \,,\\
  \Lprod{v_4}{\partial_y\varphi_4}_M = -\meas{M}/(6h) \,,\\
  \Lprod{v_5}{\partial_y\varphi_5}_M 
  =\Lprod{2b_1}{\partial_y\varphi_5}_M = -\meas{M}/(3h)\,,
\end{gather*}
and choosing $\alpha_i = -q_i h$ we obtain
\[
  \Lprod{\partial_yv^\star}{g}_M = \sum_{i=3}^{5} \Lnorm{q_i\varphi_i}{M}^2 = \Lnorm{g}{M}^2 \,\deeper.
\]
Finally, since $\abs{\alpha_i}\leq hq_i$ and $\Hseminorm{v_i}{M}^2\leq Ch^{-2}\meas{M}$
we get
\[
  \Hseminorm{v^\star}{M}^2
  \leq C\sum_{i=3}^{5}\Hseminorm{\alpha_iv_i}{M}^2
  \leq C\sum_{i=3}^{5}h^2q_i^2 (h^{-2}\meas{M})
  = C\Lnorm{g}{M}^2\,\deeper,
\]
which finishes case 1  as $(0,v^\star)\in\VP(M)$.

\underline{Case 2:} 
On appropriate triangulations of a 2-by-2 macro element one can construct orthogonal bases of the velocity and pressure spaces, again even and odd on $y$.
Then, the proof  follows   a similar way to Case 1.

\subsection{A numerical confirmation}
The results presented in Table~\ref{edgeP2P1-allowed} 
show that the Taylor--Hood pair is also uniformly stable on triangulated edge and corner patches.
Furthermore, we see that the difference between the inf-sup constants in column 1 and 2 is marginal, that is, the additional refinement does not affect the inf-sup constant.
On the other hand,  we show that hypothesis~\ref{hyp_GoodEdgePatch} for edge patches must be satisfied for their triangulated versions, as Figure~\ref{fig:EPtribad} confirms.

\begin{figure}[hbtp]
  \centering
  \includegraphics[scale=1]{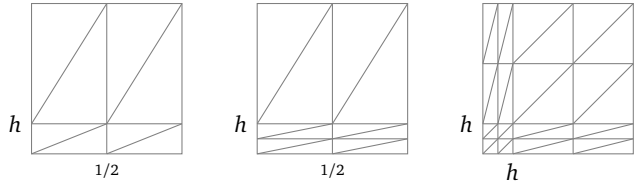}
  \caption{Triangulated edge and corner patches satisfying the hypotheses
of Theorems~\ref{thm:THEPs} and \ref{thm:THCPs}. In all cases, $\Omega=(0,1)^2$.}
  \label{fig:CPsP2P1}
\end{figure}

\begin{table}[hbtp]
\centering
\caption{Discrete inf-sup constants of the $\mathbb{P}_2^2\times\mathbb{P}_1$ pair on the partitions shown in Fig.~\ref{fig:CPsP2P1}}
\label{edgeP2P1-allowed}
\begin{tabular}{lccc}
\toprule
\multicolumn{1}{c}{$h$} 
& \multicolumn{1}{c}{Fig.~\ref{fig:CPsP2P1}(left)} 
& \multicolumn{1}{c}{Fig.~\ref{fig:CPsP2P1}(centre)} 
& \multicolumn{1}{c}{Fig.~\ref{fig:CPsP2P1}(right)} 
\\
\midrule
$10^{-1}$  &    $0.3346$ &  $0.3332$ &    $0.3844$\\
$10^{-2}$  &    $0.2690$ &  $0.2677$ &    $0.3744$\\
$10^{-3}$  &    $0.2480$ &  $0.2478$ &    $0.3519$\\
$10^{-4}$  &    $0.2453$ &  $0.2452$ &    $0.3493$\\
$10^{-5}$  &    $0.2450$ &  $0.2450$ &    $0.3491$\\
\bottomrule
\end{tabular}
\end{table}

\begin{figure}[htbp]
  \centering
  \raisebox{-.5\height}{\includegraphics{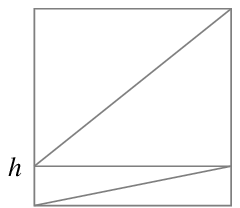}}
  \hspace{1cm}
\begin{tabular}{lc}
\toprule
\multicolumn{1}{c}{$h$} 
& \multicolumn{1}{c}{$\beta_\PO$} 
\\
\otoprule
$10^{-1}$  &  $1.788\cdot10^{-1}$ \\
$10^{-2}$  &  $8.236\cdot10^{-2}$ \\
$10^{-3}$  &  $2.725\cdot10^{-2}$ \\
$10^{-4}$  &  $8.656\cdot10^{-3}$ \\
$10^{-5}$  &  $2.739\cdot10^{-3}$ \\
\bottomrule
\end{tabular}
\caption{A triangulated  edge patch that does not satify Hypothesis~\ref{hyp_GoodEdgePatch} and the associated discrete inf-sup constants for the $\mathbb{P}_2\times\mathbb{P}_1$ pair.}
  \label{fig:EPtribad}
\end{figure}

\section{Possible extensions}
In this work we have proven the uniform inf-sup stability of the lowest order Taylor--Hood pair in a family of anisotropic meshes.
Up to our best knowledge, this is the first proof available for this pair on anisotropic meshes.
The numerical evidence shown suggests that the hypotheses made for the partitions (macro elements) are minimal, so the results presented here are optimal.

There are, nevertheless, open questions.
The first is the possible extension to higher order polynomials.
Another possible extension is the possibility to allow geometric refinements towards a corner of a given partition. 
The proof of Theorem~\ref{thm:THCPs} would not fail in this situation provided the result on geometric edge patches holds. In fact, a dependency on a level of refinement as documented in Remark~\ref{rem:CPrefinements} would not occur. 
Finally, the problem of stability of an anisotropic refinement strategy, driven by \emph{a posteriori} error estimators, is also a topic of interest.
All these constitute open questions that will be subject of future research.

%


\providecommand{\bysame}{\leavevmode\hbox to3em{\hrulefill}\thinspace}
\providecommand{\MR}{\relax\ifhmode\unskip\space\fi MR }
\providecommand{\MRhref}[2]{%
  \href{http://www.ams.org/mathscinet-getitem?mr=#1}{#2}
}
\providecommand{\href}[2]{#2}

\end{document}